\documentclass[12pt]{article}
\usepackage{amsmath} 
\usepackage{amssymb} 
\usepackage{graphicx} 
\usepackage{enumerate}
\usepackage[utf8]{inputenc}
\usepackage{amsthm}
\usepackage{color}

\newtheorem{thm}{Theorem}

\newcommand{\F}{\mathbb{F}}
\newcommand{\Fq}{\mathbb{F}_q}

\begin{document}

\title{On the Index of the Diffie-Hellman Mapping}

\date{}
\author{Leyla I\c s\i k$^1$, Arne Winterhof$^2$ }

\maketitle

\noindent
$^1$ \.{I}stinye \"{U}niversitesi Topkap{\i} Kamp\"{u}s\"{u}, Maltepe Mah., Teyyareci Sami Sk., No.3 Zeytinburnu, 34010 \.{I}stanbul, Turkey\\
E-mail: leyla.isik@istinye.edu.tr\\

\noindent
$^2$ Johann Radon Institute for Computational and Applied Mathematics\\
Austrian Academy of Sciences, Altenbergerstr.\ 69, 4040 Linz, Austria\\
E-mail: arne.winterhof@oeaw.ac.at

\begin{abstract} Let $\gamma$ be a generator of a cyclic group $G$ of order $n$. 
The least index of a self-mapping $f$ of $G$ is the index of the largest subgroup $U$ of $G$ such that $f(x)x^{-r}$ is constant on each coset of $U$ for some positive integer~$r$.  
We determine the index of the univariate Diffie-Hellman mapping $d(\gamma^a)=\gamma^{a^2}$, $a=0,1,\ldots,n-1$, and show that any mapping of small index coincides with~$d$ only on a small subset of $G$.
Moreover, we prove similar results for the bivariate Diffie-Hellman mapping $D(\gamma^a,\gamma^b)=\gamma^{ab}$, $a,b=0,1,\ldots,n-1$. 
In the special case that $G$ is a subgroup of the multiplicative group of a finite field we present improvements.
\end{abstract}

\bigskip


\bigskip


\section{Introduction}
Let $G$ be a (multiplicatively written) finite cyclic group of order $n \ge 2$, $\gamma$ be a generator of $G$
and $\ell$ be a positive divisor of~$n$. Then the set of nonzero $\ell$th powers
$$C_{\ell,0}=\left\{\gamma^{j\ell} : j=0,1,...,\frac{n}{\ell}-1\right\}$$ is a subgroup of $G$ of index $\ell$. The elements of the factor group $G /C_{0}$ are the {\it cyclotomic cosets}
\begin{equation*}
C_{\ell,i}=\gamma^{i}C_{\ell,0}, ~~ i=0,1,...,\ell-1.
\end{equation*}

For any positive integer $r$ and any $a_0, a_1,..., a_{\ell-1} \in G$, we define the {\it $r$-th order cyclotomic
mapping $f^{r}_{a_0,a_1,...,a_{\ell-1}}$ of index $\ell$} by
\begin{equation}\label{eqn:rth-GcycMap}
f^{r}_{a_0,a_1,...,a_{\ell-1}}(x)=
 a_{i}x^{r} \quad  \mbox{if $x\in C_{\ell,i}$}, \quad i=0,1,\ldots,\ell-1.
\end{equation}
For a self-mapping $f$ of $G$ we denote by $ind(f)$ the smallest index $\ell$ such that~$f$ can be represented by a mapping of the form $(\ref{eqn:rth-GcycMap})$. 

Any self-mapping of the multiplicative group $\F_q^*$ of a finite field can be uniquely represented by a polynomial over $\F_q$ of degree at most $q-1$ with $f(0)=0$. The index of any polynomial over $\Fq$ 
(with constant term $0$) introduced in \cite{agw,Wang} (which was based on \cite{NidWin05}) coincides with our definition. 
In this special case the index has raised increasing interest, see for example \cite{isikWin18}, the survey article \cite{wasurvey} and references therein.
In particular, any mapping of small index is highly predictable and a large index is needed for cryptographic functions.

The security of the Diffie-Hellman key exchange, see for example \cite[Chapter~2]{niwi}, for the group $G$ is based on the infeasibility of evaluating
the {\it (bivariate) Diffie-Hellman mapping} $D$,
\begin{equation}\label{DHeqnG}
D(\gamma^a,\gamma^b)=\gamma^{ab},\quad a,b=0,\ldots,n-1.
\end{equation}
The bivariate Diffie-Hellman mapping can be efficiently reduced to the {\it univariate Diffie-Hellman mapping},
\begin{equation}\label{DH1eqnG}
d(\gamma^a)=\gamma^{a^2},\quad a=0,\ldots,n-1,
\end{equation}
since $$D(\gamma^a,\gamma^b)^2=d(\gamma^{a+b})d(\gamma^{a})^{-1}d(\gamma^{b})^{-1}$$
and square roots in $G$ can be calculated efficiently using the {\it Tonelli-Shanks algorithm}, see for example \cite[Chapter 7]{BachSha96}.

In practice, subgroups of the multiplicative group of a finite field and elliptic curves over finite fields are mainly used. For these groups many results on 
polynomials representing and interpolating the univariate and bivariate Diffie-Hellman mapping have been obtained, in particular, lower bounds on degree and sparsity, see
\cite{blga04,blga08,cosh,elsh,kiwi04,kiwi06,lawi,meve,MeiWin02,Win01} and the monograph~\cite{ShpBook}.

In this paper, we first study the index of the univariate Diffie-Hellman mapping for a generic cyclic group of order $n$ in Section~\ref{uni}.
We show that $ind(d)$ is $n$ for odd $n$ and $n/2$ for even $n$ as well as that each mapping of small index coincides with $d$ only on a small subset of $G$.

In Section \ref{bi} we introduce the index pair of a bivariate function over $G$ and obtain similar results for the bivariate Diffie-Hellman mapping, as well.
For~$G=\F_q^*$ and $k$-variate polynomials the index $k$-tuple has already been defined in \cite{mww}.

In the special case that $G$ is a subgroup of the multiplicative subgroup $\F_q^*$ of the finite field $\F_q$ we obtain some improvements in Section~\ref{fq}. 

We will use the notation 
$$f(n)=O(g(n))\quad\mbox{if} \quad |f(n)|\le cg(n)$$ for some constant~$c>0$ and 
$$f(n)=o(g(n))\quad \mbox{if}\quad \lim\limits_{n\rightarrow\infty} \frac{f(n)}{g(n)}=0.$$
$$f(n)\ll g(n)\quad \mbox{and}\quad g(n)\gg f(n)\quad \mbox{are both equivalent to }f(n)=O(g(n)).$$

\section{Index of the univariate Diffie-Hellman mapping}\label{uni}

First we determine the index $ind(d)$ of the univariate Diffie-Hellman mapping.

\begin{thm}\label{thm:IndUniv}
Let $G$ be any cyclic group of order $n$ with generator $\gamma$. Then the index of the univariate Diffie-Hellman mapping $d$ of $G$ defined by $(\ref{DH1eqnG})$ is
 \begin{equation*}
 ind(d)=\left\{\begin{array}{cl} n, & \mbox{$n$ is odd} ,\\
                                 n/2, & \mbox{$n$ is even}.
                                \end{array}\right. 
\end{equation*}
\end{thm}

\begin{proof} Let $\ell$ denote the index of $d$, that is, 
$$d=f^r_{a_0,\ldots,a_{\ell-1}}$$
for some positive integer $r$ and $a_0,\ldots,a_{\ell-1}\in G$,
where $f^r_{a_0,\ldots,a_{\ell-1}}$ is defined by~$(\ref{eqn:rth-GcycMap})$. 
Then we have  
\begin{equation}\label{eqn:Univ1}
d(\gamma^{j\ell+i})=a_i \gamma^{r(j\ell+i)}=\gamma^{(j\ell+i)^2},\quad j=0,\ldots,\frac{n}{\ell}-1,\quad i=0,\ldots,\ell-1.
\end{equation} 
Taking $j=0$ and $j=1$ we get
\begin{equation*} 
a_i=\gamma^{-ri+i^2}=\gamma^{(\ell+i)^2-r(\ell+i)}, \quad i=0,\ldots,\ell-1,
\end{equation*} 
which implies 
\begin{equation}\label{eqn:r_cong}
r\equiv \ell+2i \bmod \dfrac{n}{\ell},  \quad i=0,\ldots,\ell-1.
\end{equation} 
 Thus either $\ell=1$ or $n/\ell$ divides $2$. 
 
 If $\ell=1$, note that $r\equiv\ell\equiv1 \bmod n$ by~$(\ref{eqn:r_cong})$.
 Then $(\ref{eqn:Univ1})$ applied with $j=0$ and $j=n-1$ implies $\gamma^{-1}=\gamma$ and thus $n\in \{1,2\}$. 
 
 If $n/\ell$ divides $2$, we have $\ell=n$ if $n$ is odd and $\ell=n/2$ or $\ell=n$ if $n$ is even.
It remains to show that for even $n$, $d$ can be represented by a mapping of index~$n/2$. 

Suppose that $n$ is even and $\ell=n/2$, which means that each coset $C_{\ell,i}$ of~$G$ contains only two elements, $\gamma^{i}$ and $\gamma^{i+n/2}$, for $i=0,\ldots,n/2-1$. 
Choose any~$r$ with $r\equiv \frac{n}{2} \bmod  2$ and $a_i=\gamma^{i^2-ir}$. Then it is easy to verify that 
\begin{equation*}
d(\gamma^{i})=\gamma^{i^2}=a_i \gamma^{ir}
\quad \mbox{and}\quad 
d(\gamma^{i+n/2})=\gamma^{(i+n/2)^2}=a_i \gamma^{(i+n/2)r} 
\end{equation*} 
for $i=0,1,\ldots,n/2-1$
and the result follows.
\end{proof}

Theorem~\ref{thm:IndUniv} states only that the univariate Diffie-Hellman mapping $d$ cannot coincide with a mapping of small index in {\it all} points. 
However, by the following result it cannot even coincide in {\it many} points.

\begin{thm}\label{two}
 The univariate Diffie-Hellman mapping $d$ of the cyclic group $G$ of order~$n$ coincides with any mapping of index $\ell$
 in   
 $$O(\ell n^{1/2})$$ 
 elements of $G$. If $n$ is prime, we have the better bound $2\ell$. 
\end{thm}
\begin{proof}
 For fixed $a\in \{0,1,\ldots,n-1\}$ consider the mapping $f_a(y)=\gamma^ay^r$, $y\in G$.
 We have to estimate the number $N$ of $x=0,1,\ldots,n-1$ with 
 $$f_a(\gamma^x)=d(\gamma^x),$$
 that is,
 $$\gamma^{a+rx}=\gamma^{x^2},$$
 or equivalently,
 $$x^2-rx-a\equiv 0\bmod n.$$
 By \cite{kon} we have
 $$N=O(n^{1/2})$$
 for any $n$. If $n$ is prime, we have obviously $N\le 2$.
 Since each function of index~$\ell$ is the combination of at most $\ell$ different functions of the form $f_a$ and the result follows. 
\end{proof}

\section{Index of the bivariate Diffie-Hellman mapping}\label{bi}

Let $\ell_1$ and $\ell_2$ be divisors of $n$ and $G$ the cyclic group of order $n$.

For any positive integers $r_1$ and $r_2$ and any $a_{0,0},\dots,a_{\ell_{1}-1,\ell_{2}-1} \in G$, we define the {\it $(r_{1},r_{2})$th order cyclotomic
mapping $f^{(r_1,r_2)}_{a_{0,0},\ldots,a_{\ell_{1}-1,\ell_{2}-1}}$ of index pair $(\ell_1,\ell_2)$} by
\begin{equation}\label{eqn:rth-BivcycMap}
f^{(r_1,r_2)}_{a_{0,0},\ldots,a_{\ell_{1}-1,\ell_{2}-1}}(x,y)=
                                            a_{k_1,k_2}x^{r_1}y^{r_2} \quad \mbox{if $(x,y)\in C_{\ell_1,k_1} \times C_{\ell_2,k_2}$},
\end{equation}
for $k_1=0,\ldots,\ell_{1}-1$ and $k_2=0,\ldots,\ell_{2}-1$. For a mapping $f$ over $G$ with the property (\ref{eqn:rth-BivcycMap}) we call $(\ell_1,\ell_2)$  an {\it index pair of $f$.}

\begin{thm} 
Let $G$ be any cyclic group of order $n$. Then the bivariate Diffie-Hellman mapping $D$ of $G$ defined by $(\ref{DHeqnG})$ has the only index pair $(n,n)$. 
\end{thm}
\begin{proof} Since otherwise the result is trivial we may assume $n\ge 2$, $\min\{\ell_1,\ell_2\}<n$ and wlog.\ $\ell_1\ge \ell_2$. 

Let $(\ell_{1},\ell_2)$ be an index pair of $D$, that is, $D$ can be represented by a mapping of the form $(\ref{eqn:rth-BivcycMap})$.
Then
\begin{equation}\label{eqn:biv1}
D(\gamma^{k_1+j_1\ell_1},\gamma^{k_2+j_2\ell_2})=\gamma^{(k_1+j_1\ell_1)(k_2+j_2\ell_2)}=a_{k_1,k_2}\gamma^{r_1(k_1+j_1\ell_1)}\gamma^{r_2(k_2+j_2\ell_2)}
\end{equation}
for $j_1=0,\ldots,n/\ell_1-1$, $j_2=0,\ldots,n/\ell_2-1$, $k_1=0,\ldots,\ell_1-1$ and $k_2=0,\ldots,\ell_2-1$.

Taking $j_1=j_2=0$ we get 
\begin{equation} \label{eqn:a1}
a_{k_1,k_2}=\gamma^{k_1k_2-r_1k_1-r_2k_2}.
\end{equation} 
Taking $j_1=0$ and $j_2=1$ gives
\begin{equation} \label{eqn:a2}
a_{k_1,k_2}=\gamma^{k_1k_2+k_1\ell_2-r_1k_1-r_2k_2-r_2\ell_2}.
\end{equation} 
Combining $(\ref{eqn:a1})$ and $(\ref{eqn:a2})$ yields 
\begin{equation*}
r_2  \equiv k_1 \bmod \frac{n}{\ell_2}
\end{equation*}
for $k_1=0,\ldots,\ell_1-1$. Thus $\ell_1=1$ and also $\ell_2=1$ by our assumption $\ell_2\le \ell_1$.

Since $\ell_1=\ell_2=1$, we have $k_1=k_2=0$ and $r_2\equiv 0\bmod n$. Then $(\ref{eqn:biv1})$ becomes
$$D(\gamma^{j_1},\gamma^{j_2})=\gamma^{j_1j_2}=a_{0,0}\gamma^{r_1j_1}$$
and thus
$$a_{0,0}=\gamma^{j_1j_2-r_1j_1}.$$
Taking $j_1=0$ and $j_1=1$, respectively, we get 
$$a_{0,0}=1=\gamma^{j_2-r_1},$$
that is,
$$j_2\equiv r_1 \bmod n$$
for $j_2=0,\ldots, n-1$. This is not possible unless $n=1$ which contradicts our assumption. 
%
%
\end{proof}

\begin{thm}\label{four}
 Any mapping of index pair $(\ell_1,\ell_2)$ coincides with the bivariate Diffie-Hellman mapping $D$ of the cyclic group $G$ of order $n$ in 
 at most $n^{1+o(1)}\ell_1\ell_2$ elements of $G^2$. 
\end{thm}
\begin{proof}
 For each $\gamma^a\in G$ the mapping $f_a(\gamma^x,\gamma^y)=\gamma^a\gamma^{r_1x}\gamma^{r_2y}$ coincides with $D(\gamma^x,\gamma^y)=\gamma^{xy}$ if and only if 
 $$xy\equiv a+r_1x+r_2y\bmod n.$$
 For fixed $y$ put $t=\gcd(y-r_1,n)$. If $t$ does not divide $a+r_2y$, there is no solution $x$. Otherwise the equation is equivalent to 
 $$x\frac{y-r_1}{t}\equiv \frac{a+r_2y}{t} \bmod \frac{n}{t},$$
 which has a unique solution $x$ modulo $n/t$, that is, $t$ solutions modulo $n$.
 For each $t$ there are $\varphi(n/t)$ different $y\in \{0,\ldots,n-1\}$ with $\gcd(y-r_1,n)=t$, where $\varphi$ is Euler's totient function. Hence, we have
 $$\sum_{t|n}\varphi(n/t)t=n\sum_{d|n}\frac{\varphi(d)}{d}\le \tau(n)n=n^{1+o(1)}$$
 solutions, where $\tau(n)=n^{o(1)}$ is the number of divisors of $n$.
 Therefore each mapping of index pair $(\ell_1,\ell_2)$ coincides with $D$ in at most $\ell_1\ell_2n^{1+o(1)}$ elements of $G^2$.
\end{proof}

\section{Multiplicative subgroups of finite fields}\label{fq}

In this section let $G$ be a subgroup of $\F_q^*$ of order $n|q-1$ and $\gamma\in \F_q^*$ be of order $n$.
First we deal with the univariate case.

\begin{thm}\label{five}
 Let $f$ be any self-mapping of $\F_q^*$ satisfying 
 $$f(\gamma^x)=\gamma^{x^2},\quad x\in S,$$
 for a subset $S\subseteq \{N+1,\ldots,N+H\}$ of cardinality $|S|=H-s$ with $H\le n$.
 Then we have 
 $$ind(f)\ge \frac{n}{2(n-H+2s+1)}.$$
\end{thm}
\begin{proof}
  For $H=n$ and $s=0$ the result follows from Theorem~\ref{thm:IndUniv} and we may restrict ourselves to the case $n-H+2s+1\ge 2$. 
  Since otherwise the result is trivial we may also assume  
  $$ind(f)\le n/3.$$
  A straightforward extension of \cite[Theorem 1]{NidWin05} provides that any mapping $G$ of index $\ell$ can be represented by a polynomial of the form
  \begin{equation}\label{GX} G(X)=X^r\sum_{i=0}^{\ell-1}A_i X^{in/\ell}.
  \end{equation}
  Now assume that $f$ is of index $\ell$ and thus $h$ defined by
  $$h(\gamma^x)=f(\gamma^x)\gamma^{-rx},\quad x=0,\ldots,n-1,$$
  can be uniquely represented as 
  $$h(X)=G(X)X^{-r}$$
  for some positive integer $r$ and polynomial $G(X)$ of the form $(\ref{GX})$.
  In particular, the weight $w(h)$, that is, the number of nonzero coefficients of $h(X)$, is at most $\ell$, and the degree of $h(X)$ at most $(\ell-1)n/\ell\le n-3$.
  For all but at most $s+1$ elements $x$ of $S$ we have
  \begin{eqnarray*} h(\gamma^{x+1})&=&f(\gamma^{x+1})\gamma^{-r(x+1)}=\gamma^{(x+1)^2-r(x+1)}\\
   &=&\gamma^{x^2-rx}(\gamma^x)^2\gamma^{1-r}
  =\gamma^{1-r}(\gamma^x)^2h(\gamma^x).
  \end{eqnarray*}
  Hence, the polynomial 
  $$F(X)=h(\gamma X)-\gamma^{1-r}X^2h(X)$$
  has at least $|S|-s-1=H-2s-1$ zeros of the form $\gamma^x$, $x\in \{1,\ldots,n\}$.
  The weight $w(F)$ of $F(X)$ satisfies
  $$w(F)\ge \frac{n}{n-H+2s+1}$$
  by \cite[Lemma~1]{lawi02}, which is applicable since $\deg(F)\le n-1$.
  On the other hand, $w(F)\le 2w(h)$ and thus
  $$\ell\ge w(h)\ge \frac{n}{2(n-H+2s+1)},$$
  which completes the proof.
\end{proof}

Remark. 
Theorem~\ref{two} implies 
$$ind(f)\gg \frac{|S|}{n^{1/2}}$$
for any $S$. This lower bound does not exceed $n^{1/2}$.
However, Theorem~\ref{five} provides a larger lower bound than $n^{1/2}$ for any $S$ satisfying the conditions of Theorem~\ref{five} with $n-|S|=o(n^{1/2})$. 

Similar ideas can be used to prove an analog of Theorem~\ref{five} for the bivariate Diffie-Hellman mapping. \\
\begin{thm}\label{six}
 Let $G$ be a subgroup of $\F_q^*$ of order $n|q-1$ generated by $\gamma$, $U$ be any subset of $\{0,1,\ldots,n-1\}$
and $V=\{N,\ldots,N+H-1\}$ be any set of consecutive integers for some $H\le n$. Let $f:G\times G\rightarrow G$ 
be any mapping of index pair $(\ell_1,\ell_2)$ satisfying
$$f(\gamma^x,\gamma^y)=\gamma^{xy},\quad (x,y)\in U\times V.$$
Then we have 
$$\max\{\ell_1,\ell_2\}\ge \min\{|U|,H\}.$$
\end{thm}
\begin{proof}
 Put $m=\min\{|U|,H\}$.

It is easy to see that any mapping $f$ of index pair $(\ell_1,\ell_2)$ and order $(r_1,r_2)$ can be represented by a polynomial 
$f(X,Y)$ over $\F_q$
of the form
$$f(X,Y)=X^{r_1}Y^{r_2}\sum_{i=0}^{\ell_1-1}\sum_{j=0}^{\ell_2-1}a_{i,j}X^{in/\ell_1}Y^{jn/\ell_2}.$$
 Then there is a subset $\{u_0,\ldots,u_{m-1}\}$ of $U$ such that
$$\gamma^{u_x(N+y)-r_1u_x-r_2(N+y)}=\sum_{i=0}^{\ell_1-1}\sum_{j=0}^{\ell_2-1}a_{i,j}\gamma^{inu_x/\ell_1+jn(N+y)/\ell_2},\quad
x,y=0,\ldots,m-1.$$
Assume $\max\{\ell_1,\ell_2\}< m$.
Then the coefficient matrix $A=(a_{i,j})_{i,j=0,\ldots,m-1}$, with $a_{i,j}=0$ if $i\ge \ell_1$ or $j\ge \ell_2$, 
satisfies
$$G=V_1AV_2,$$
where 
$$V_1=\left(\gamma^{inu_x/\ell_1}\right)_{i,x=0,\ldots,m-1},\quad V_2=\left(\gamma^{jn(N+y)/\ell_2}\right)_{y,j=0,\ldots,m-1}$$
and 
$$G=\left(\gamma^{(u_x-r_2)(N+y)-r_1u_x}\right)_{x,y=0,\ldots,m-1}.$$
$V_1$ and $V_2$ are Vandermonde matrices and $G$ can be reduced to a Vandermonde matrix by multiplying the $x$th row by the constant $\gamma^{r_1u_x}$.
 Hence, $A$ is the product
of three invertible matrices 
$$A=V_1^{-1}GV_2^{-1}$$
and thus invertible itself.
In particular, each row and each column of $A$ contains at least one nonzero entry which contradicts our assumption $\max\{\ell_1,\ell_2\}<m$.
\end{proof}

Remark. Theorem~\ref{four} implies the lower bound 
$$\max\{\ell_1,\ell_2\}\ge (\ell_1\ell_2)^{1/2}\ge \left(\frac{|U|H}{n^{1+o(1)}}\right)^{1/2}.$$
Its right hand side is always smaller than $n^{1/2}$. Theorem~\ref{six} 
provides a lower bound $\ge n^{1/2}$ for any $U$ and $H$ satisfying $\min\{|U|,H\}\ge n^{1/2}$.

\section*{Acknowledgment}

The second author is partially supported by the Austrian Science Fund FWF Project P 30405-N32. 
Parts of this paper were written during a visit of the first
author to RICAM. She would like to express her sincere thanks for the hospitality during her visit.
The authors would like to thank Steven Wang for useful discussions.

\end{document}